\documentclass[10pt,dvipsnames]{amsart}
\usepackage{hyperref}
\usepackage{amssymb, amscd, amsmath, amsfonts, amsthm}
\usepackage{stmaryrd}
\usepackage{verbatim}
\usepackage{multicol}

\def\unprotectedboldentry#1{\textcolor{Red}{\textbf{#1}}}
\def\boldentry{\protect\unprotectedboldentry}
\usepackage{tikz}
\usetikzlibrary{calc}
\usepackage{ifthen}
\newcommand{\tikztableau}[2][scale=0.6,every node/.style={font=\small}]{
    \def\newtableau{#2}
    \begin{array}{c}
    \begin{tikzpicture}[#1]
    \coordinate (x) at (-0.5,0.5);
    \coordinate (y) at (-0.5,0.5);
    \foreach \row in \newtableau {
        \coordinate (x) at ($(x)-(0,1)$);
        \coordinate (y) at (x);
        \foreach \entry in \row {
            \ifthenelse{\equal{\entry}{X}}
               {
                \node (y) at ($(y) + (1,0)$) {};
                \fill[color=gray!10] ($(y)-(0.5,0.5)$) rectangle +(1,1);
                \draw[color=gray] ($(y)-(0.5,0.5)$) rectangle +(1,1);
               }
               {
                \ifthenelse{\equal{\entry}{\boldentry X}}
                   {
                    \node (y) at ($(y) + (1,0)$) {};
                    \fill[color=gray] ($(y)-(0.5,0.5)$) rectangle +(1,1);
                    \draw ($(y)-(0.5,0.5)$) rectangle +(1,1);
                   }
                   {
                    \node (y) at ($(y) + (1,0)$) {\entry};
                    \draw ($(y)-(0.5,0.5)$) rectangle +(1,1);
                   }
               }
            }
        }
    \end{tikzpicture}
    \end{array}}

\newcommand\kschur[1][k]{s^{(#1)}}

\def\sym{\operatorname{\mathsf{Sym}}}
\def\symk{\operatorname{\mathsf{Sym}_{(k)}}}
\def\symkk{\operatorname{\mathsf{Sym}}^{(k)}}
\def\Qsym{\operatorname{\mathsf{QSym}}}
\def\Nsym{\operatorname{\mathsf{NSym}}}
\def\Qsymk{\operatorname{\mathsf{QSym}}^{(k)}}
\def\Nsymk{\operatorname{\mathsf{NSym}}_{(k)}}
\def\P{\operatorname{\rm{\mathcal{P}}}}
\def\Pk{\operatorname{\rm{\mathcal{P}}^k}}
\def\C{\operatorname{\rm{\mathcal{C}}}}
\def\Ck{\operatorname{\rm{\mathcal{C}}^k}}
\newcommand\QS[1][\alpha]{\rm{S}^\star_{#1}}
\newcommand\NS[1][\alpha]{\rm{S}_{#1}}
\newcommand\QSk[1][\alpha]{\rm{S^{(k) \star}_{#1}}}
\newcommand\NSk[1][\alpha]{\rm{S^{(k)}_{#1}}}

\newcommand\dualkschur[1][\lambda]{\rm{{s}}^{(k) \star}_{#1}}

\newtheorem{Theorem}{Theorem}[section]

\newtheorem{Corollary}[Theorem]{Corollary}

\newtheorem{Example}[Theorem]{Example}
\newtheorem{Remark}[Theorem]{Remark}
\newtheorem{Definition}[Theorem]{Definition}

\allowdisplaybreaks

\begin{document}

\title{Quasi-symmetric and non-commutative Affine Schur Functions}
\author{Chris Berg \and Luis Serrano}
\date{\today}

\begin{abstract}
We introduce dual Hopf algebras which simultaneously combine the concepts of the $k$-Schur function theory of \cite{LLM03, LM08, Lam06} with the quasi-symmetric Schur function theory of \cite{HLMvW, BLvW}. Mimicking their procedures, we construct dual basis of these Hopf algebras with remarkable properties.
\end{abstract}
\maketitle

\section{Introduction}

The $k$-Schur functions $s^{(k)}_\lambda(X;t)$ of Lapointe, Lascoux and Morse \cite{LLM03} first arose in the study of Macdonald polynomials. Since then, Lam \cite{Lam08} verified a conjecture of Morse and Shimozono, which realized $s^{(k)}_\lambda(X,1)$ as the Schubert basis for the homology of the affine Grassmannian. Lapointe and Morse \cite{LM08} also verified that the structure coefficients for these polynomials gave 3-point Gromov-Witten invariants and fusion coefficients in relevant cases. On the combinatorial side, many analogous notions to the combinatorics of Schur functions were developed, some of which we discuss here as motivation. The space of $k$-Schur functions has a Hopf structure. The corresponding dual Hopf algebra was first studied by Lapointe and Morse \cite{LM08} and Lam \cite{Lam06}. The dual $k$-Schur functions similarly play a role in the cohomology of the affine Grassmannian.

More recently, Haglund, Luoto, Mason and van Willigenburg \cite{HLMvW} developed the notion of quasi-symmetric Schur functions for the algebra $\Qsym$. This is a new basis of $\Qsym$ with many remarkable properties generalizing the notion of Schur functions. The space $\Nsym$ of non-commutative symmetric functions is the Hopf dual to $\Qsym$. The functions dual to quasi-symmetric Schur functions will here by called non-commutative Schur functions. Of particular importance to us is the non-commutative Pieri rule discussed in \cite{BLvW}.

The purpose of this paper is to bring the two ideas of $k$-Schur functions and quasi-symmetric Schur functions together. Using only the basic concepts from both subjects, we introduce two new Hopf algebras $\Qsymk$ and $\Nsymk$ which are dual to one another and build Schur like bases for both spaces. In Section \ref{sec:schur} we use the example of Schur functions to motivate our construction. In Section \ref{sec:kschur}, we remind the reader of $k$-Schur functions by applying the same construction in the context of $k$-Schur functions. In Section \ref{sec:quasi}, we do the same in the context of quasi-symmetric Schur functions. Finally, in Section \ref{sec:kquasi}, we define and develop our new theory.

\section{Schur functions and the self dual Hopf algebra $\sym$}\label{sec:schur}

In this section, we recall basic results concerning symmetric functions. We refer the reader to \cite{Sta99}. We discuss these classical results only to develop an analogy with $k$-Schur functions, quasi-symmetric Schur functions, and ultimately quasi-symmetric affine Schur functions.

We let $\P$ denote the set of all partitions and $\sym$ denote the algebra of symmetric functions. 
The $i^{th}$ complete homogenous symmetric function $h_i$ is the symmetric function which contains all monomials of degree $i$. It is well known that the $h_i$ are algebraically independent and generate $\sym$, i.e. $\sym = \mathbb{Q}[h_1, h_2, \dots]$. With this description, a natural basis for $\sym$ is $h_\lambda := h_{\lambda_1} h_{\lambda_2} \cdots h_{\lambda_m}$ over all partitions $\lambda = (\lambda_1, \lambda_2, \dots, \lambda_m) \in \P$.

There exists another basis of $\sym$, called the monomial symmetric functions $\{m_\lambda\}_{\lambda \in P}$, defined by, \[m_\lambda = \sum_{\alpha} x^\alpha,\] where the sum is over all distinct $\alpha$ which rearrange to form $\lambda$.
$\sym$ has a coproduct, making it into a Hopf algebra. 

\begin{Theorem} 
The Hopf algebra $\sym$ is self dual. Under this duality, the basis $\{m_\lambda\}$ is dual to $\{h_\lambda\}$.
\end{Theorem}

We let $\langle \cdot, \cdot \rangle_{\P}$ denote the scalar product on $\sym$ defined by $\langle m_\lambda, h_\mu \rangle_{\P} = \delta_{\lambda, \mu}$. 

\subsection{Horizontal strips and Schur functions}

For two partitions $\mu \subset \lambda$, we say that $\lambda / \mu$ is a horizontal strip of size $i$ if the corresponding diagram contains $i$ cells, no two of which are in the same column.

A semistandard tableaux of shape $\lambda$ and content $\mu$ is a sequence of partitions $\emptyset = \nu^0 \subset \nu^1 \subset \nu^2 \subset \cdots \subset \nu^m = \lambda$ such that the skew diagrams $\nu^i / \nu^{i-1}$ are horizontal strips of size $\mu_i$. The number of semistandard tableaux of shape $\lambda$ and content $\mu$ is denoted $K_{\lambda, \mu}$. 

\begin{Definition}
The Schur functions $s_\lambda$ are the basis of $\sym$ uniquely defined by the Pieri rule: 
\[ h_i s_\lambda  = \sum_\mu s_\mu,\] where the sum is over all partitions $\mu$ for which $\mu / \lambda$ is a horizontal strip of size $i$.
\end{Definition}

\begin{Corollary}\label{rem:schur}
For a partition $\mu$, \[h_\mu = \sum_{\lambda} K_{\lambda, \mu} s_\lambda .\]
\end{Corollary}

\begin{proof}
This follows from induction and the Pieri rule.
\end{proof}

\begin{Theorem}\label{thm:selfdual}
The Schur functions $\{ s_\lambda\}$ are their own dual basis. In other words, $\langle s_\lambda, s_\mu\rangle_{\P} = \delta_{\lambda, \mu}$.
\end{Theorem}

The self duality of the Schur functions allows us to determine the monomial expansion of a Schur function. For our purposes, one should think of the monomial expansion as being dual to the Pieri rule.

\begin{Theorem}
The Schur function expands in the monomial basis as: \[s_\lambda = \sum_{\mu} K_{\lambda, \mu} m_\mu. \]
\end{Theorem}

\begin{proof}
Using  Corollary \ref{rem:schur} and Theorem \ref{thm:selfdual}, \[ \langle s_\lambda, h_\mu \rangle_{\P} = \langle s_\lambda, \sum_{\mu} K_{\nu, \mu} s_\nu \rangle_{\P} = K_{\lambda, \mu},\] so the coefficient of $m_\mu$ in $s_\lambda$ is $K_{\lambda, \mu}$.
\end{proof}

\section{$k$-Schur functions and the dual Hopf algebras $\symkk$ and $\symk$}\label{sec:kschur}

In this section, we recall basic results concerning $k$-Schur functions. We refer the reader to \cite{LLM03, LM03, LM05, Lam06, LM07,  LS07, LM08,LLMS10, Lam10}.

We fix $k \geq 1$ and let $\Pk$ denote the set of all $k$-bounded partitions. 
We let $\symk = \mathbb{Q}[h_1, h_2, \dots, h_k] \subset \sym$. A natural basis for $\symk$ is $h_\lambda := h_{\lambda_1} h_{\lambda_2} \cdots h_{\lambda_m}$ for all $k$-bounded partitions $\lambda = (\lambda_1, \lambda_2, \dots, \lambda_m) \in \Pk$.

We let $\symkk := \sym /\langle m_\lambda: \lambda \not \in \Pk\rangle$. The simple basis for this space is the collection of $k$-bounded monomials $\{m_\lambda \}_{\lambda \in \Pk}$. 

$\symk$ and $\symkk$ inherit a Hopf algebra structure from $\sym$.

\begin{Theorem} 
The Hopf algebras $\symkk$ and $\symk$  are dual. Under this duality, the basis $\{m_\lambda\}_{\lambda \in \Pk}$ is dual to $\{h_\lambda\}_{\lambda \in \Pk}$.
\end{Theorem}

The pairing $\langle \cdot, \cdot \rangle_{\Pk}: \symkk \times \symk \rightarrow \mathbb{Q}$ is defined by $\langle m_\lambda, h_\mu \rangle_{\Pk} = \delta_{\lambda, \mu}$. 

\subsection{Horizontal $k$-strips and $k$-Schur functions}

There exists an involution $\omega_k$ on the set of $k$-bounded partitions sending $\lambda$ to  $\lambda^{\omega_k}$. We will not describe it here, but it was first defined in \cite{LLM03}.

For two $k$-bounded partitions $\mu \subset \lambda$, we say that $\lambda / \mu$ is a horizontal $k$-strip if $\lambda/ \mu$ is a horizontal strip and $\lambda^{\omega_k}/\mu^{\omega_k}$ is a vertical strip (no two cells in same row).

A semistandard $k$-tableaux of shape $\lambda$ and content $\mu$ is a sequence of $k$-bounded partitions $\emptyset = \nu^0 \subset \nu^1 \subset \nu^2 \subset \cdots \subset \nu^m = \lambda$ such that the skew diagrams $\nu^i / \nu^{i-1}$ are horizontal $k$-strips of size $\mu_i$. The number of semistandard $k$-tableaux of shape $\lambda$ and content $\mu$ is denoted $K^{(k)}_{\lambda, \mu}$. 

\begin{Definition}\label{def:schur}
The $k$-Schur functions $s^{(k)}_\lambda$ are the basis of $\symk$ uniquely defined by the $k$-Pieri rule: 
\[ h_i s^{(k)}_\lambda  = \sum_\mu s^{(k)}_\mu,\]  the sum over all partitions $\mu$ for which $\mu / \lambda$ is a horizontal $k$-strip of size $i$.
\end{Definition}

\begin{Remark}
This is not the way that $k$-Schur functions were originally defined. There are in fact many conjecturally equivalent definitions of $k$-Schur functions.
\end{Remark}

\begin{Corollary}\label{rem:htokschur}
As a consequence of the $k$-Pieri rule, \[h_\mu = \sum_{\lambda} K^{(k)}_{\lambda, \mu} s^{(k)}_\lambda .\]
\end{Corollary}

\begin{Theorem}
As $k$ approaches infinity, \[ s_\lambda^{(k)} \to s_\lambda.\] 
\end{Theorem}

\begin{proof}
This follows from the definition, since as $k$ approaches infinity, the condition of a horizontal $k$-strip becomes the condition for a horizontal strip.
\end{proof}

\begin{Definition}
The dual $k$-Schur functions $\{\dualkschur\}_{\lambda \in \Pk}$ are the basis of $\symkk$ which are dual to the $k$-Schur functions.
\end{Definition}

\begin{Theorem}[Lam \cite{Lam06}, Lapointe-Morse \cite{LM08}]
The dual $k$-Schur functions expand in the monomial basis as: \[\dualkschur = \sum_{\mu} K^{(k)}_{\lambda, \mu} m_\mu.\]
\end{Theorem}

\begin{proof}
Using  Corollary \ref{rem:htokschur}, \[ \langle \dualkschur, h_\mu \rangle_{\Pk} = \langle \dualkschur, \sum_{\mu} K^{(k)}_{\nu, \mu} s^{(k)}_\nu \rangle_{\Pk} = K^{(k)}_{\lambda, \mu},\] so the coefficient of $m_\mu$ in $\dualkschur$ is $K^{(k)}_{\lambda, \mu}$.
\end{proof}

\section{quasi-symmetric Schur functions and the dual Hopf algebras $\Qsym$ and $\Nsym$}\label{sec:quasi}

In this section, we recall basic results concerning quasi-symmetric Schur functions. We refer the reader to \cite{HLMvW, HLMvW2, BLvW}.

We let $\C$ denote the set of all compositions.
Let $\{H_i\}_{i \geq 0}$ denote a collection of generators with no relations, and construct the algebra
 $\Nsym = \mathbb{Q}[H_1, H_2, \dots]$ (we refer the reader to \cite{GKLLRT}). A natural basis for $\Nsym$ is $H_\alpha := H_{\alpha_1} H_{\alpha_2} \cdots H_{\alpha_m}$ over all compositions $\alpha = [\alpha_1, \alpha_2, \dots, \alpha_m] \in \C$.

We let \[M_\alpha = \sum_{i_1 < i_2 < \cdots < i_m} x_{i_1}^{\alpha_1} x_{i_2}^{\alpha_2} \cdots x_{i_m}^{\alpha_m}\] denote the monomial quasi-symmetric function. The algebra
$\Qsym$ can be defined as being the algebra with basis $\{M_\alpha\}_{\alpha\in\C}$. The algebra $\Qsym$ was first developed in \cite{G84, GR, Stan84}, and its study has since flourished.

\begin{Theorem} 
The Hopf algebras $\Qsym$ and $\Nsym$  are dual. Under this duality, the basis $\{M_\alpha\}_{\alpha \in \C}$ is dual to $\{H_\alpha\}_{\alpha \in \C}$.
\end{Theorem}

Let $\langle \cdot, \cdot \rangle_{\C}: \Qsym \times \Nsym \rightarrow \mathbb{Q}$ denote the pairing defined by $\langle M_\alpha, H_\beta \rangle_{\C} = \delta_{\alpha, \beta}$. 

\subsection{Horizontal composition strips and quasi-symmetric Schur functions}

We now introduce some combinatorics to describe the non-symmetric Pieri rule which was developed in \cite{BLvW}.

There is an ordering of compositions defined in \cite{HLMvW}. It can be described by the covering relation $\beta \lessdot_{\C} \alpha$ whenever either $\beta$ is $\alpha$ prepended with a $1$, or $\alpha$ is $\beta$ with the first part of size $m$ in $\beta$ replaced in $\alpha$ by $m+1$ for some $m \in \mathbb{N}$.

For two compositions $\beta, \alpha$ with $\beta \lessdot_{\C} \alpha$, $\alpha //  \beta$ denotes the collection of cells in the diagram of $\alpha$ which are not in the diagram of shape $\beta$, where $\beta$ is thought of as sitting in the lower left corner of $\alpha$.
We say that $\alpha // \beta$ is a horizontal composition strip if $\beta \lessdot_{\C} \alpha$ and $\alpha // \beta$ has no two cells in the same column.

A semistandard composition tableaux of shape $\alpha$ and content $\beta$ is a sequence of compositions $\emptyset = \gamma^0 \lessdot_{\C} \gamma^1 \lessdot_{\C} \gamma^2 \lessdot_{\C} \cdots \lessdot_{\C} \gamma^m = \alpha$ such that the skew diagrams $\gamma^i // \gamma^{i-1}$ are horizontal composition strips of size $\beta_i$. The number of semistandard composition tableaux of shape $\alpha$ and content $\beta$ is denoted $\widetilde{K}_{\alpha, \beta}$. 

\begin{Definition}[\cite{BLvW}, Corollary 3.8]\label{def:ncschur}
The non-commutative Schur functions $\NS$ are the basis of $\Nsym$ uniquely defined by the non-commutative Pieri rule: 
\[ H_i \NS  = \sum_\beta \NS[\beta],\] where the sum is over all compositions $\beta$ for which $\beta // \alpha$ is a horizontal composition strip of size $i$.
\end{Definition}

\begin{Corollary}\label{rem:htoncschur}
As a consequence of the non-commutative Pieri rule, \[H_\beta = \sum_{\lambda} \widetilde{K}_{\alpha, \beta} \NS .\]
\end{Corollary}

\begin{Definition}
The quasi-symmetric Schur functions $\{\QS\}_{\alpha \in \C}$ are the basis of $\Qsym$ which are dual to the basis of non-commutative Schur functions.
\end{Definition}

\begin{Theorem}[\cite{HLMvW}, Theorem 6.1]
The quasi-symmetric Schur functions expand in the quasi-symmetric monomial basis as: \[\QS = \sum_{\beta} \widetilde{K}_{\alpha, \beta} M_\beta.\]
\end{Theorem}

\begin{Remark}
In \cite{HLMvW} the authors use the notation $S_\alpha$ for a quasi-symmetric Schur function. To make a better analogy, we use $S_\alpha$ to denote the corresponding dual basis element in $\Nsym$.
\end{Remark}

\begin{proof}
Using  Corollary \ref{rem:htoncschur}, \[ \langle \QS, H_\beta \rangle_{\C} = \langle \QS, \sum_{\gamma} \widetilde{K}_{\gamma, \beta} \NS[\gamma] \rangle_{\C} = \widetilde{K}_{\alpha, \beta},\] so the coefficient of $M_\beta$ in $\QS$ is $\widetilde{K}_{\alpha, \beta}$.
\end{proof}

There is a map from compositions to partitions, denoted $\lambda$ by abuse of notation, which sends a composition $\alpha$ to the sorted partition $\lambda(\alpha)$, and there is a corresponding projection
 $\chi: \Nsym \to \sym$, defined by sending $H_\alpha \to h_{\lambda(\alpha)}$. An amazing fact about non-commutative Schur functions is that they are lifts of the ordinary Schur functions under $\chi$.

\begin{Theorem}\label{projection} For a composition $\alpha$, 
\[ \chi(\NS) = s_{\lambda(\alpha)}.\]
\end{Theorem}

\begin{proof}
This follows from induction, Definition \ref{def:ncschur} and the fact that $\alpha//\beta$ being a horizontal composition strip implies that $\lambda(\alpha)/ \lambda(\beta)$ is a horizontal strip.
\end{proof}

The dual identity to Theorem \ref{projection} is a decomposition of a Schur function into quasi-symmetric Schur functions.

\begin{Theorem}\label{decompose}
For a partition $\lambda$, \[ s_\lambda = \sum_{\substack{\alpha \\ \lambda(\alpha) = \lambda}} \QS.\]
\end{Theorem}

\begin{proof}
We note that if $f \in \Qsym$ is actually in $\sym$ and $g \in \Nsym$ then $\langle f, g \rangle_{\C} = \langle f, \chi(g) \rangle_{\P}$.
We can compute $\langle s_\lambda, \NS \rangle_{\C} = \langle s_\lambda, s_{\lambda(\alpha)}\rangle_{\P} = \delta_{\lambda, \lambda(\alpha)}$, so the quasi-symmetric $\QS$ function appears in $s_\lambda$ with coefficient $\delta_{\lambda, \lambda(\alpha)}$.
\end{proof}

\section{quasi-symmetric affine Schur functions and the dual Hopf algebras $\Qsymk$ and $\Nsymk$}\label{sec:kquasi}

We let $\Ck$ denote the set of all $k$-bounded compositions. These are compositions where no part is bigger than $k$.

We define the algebra
 $\Nsymk = \mathbb{Q}[H_1, H_2, \dots, H_k] \subset \Nsym$. A natural basis for $\Nsymk$ is $H_\alpha := H_{\alpha_1} H_{\alpha_2} \cdots H_{\alpha_m}$ over all $k$-bounded compositions $\alpha = [\alpha_1, \alpha_2, \dots, \alpha_m] \in \Ck$.

We let
$\Qsymk$ denote the algebra $\Qsym /\langle M_\alpha : \alpha \not \in \Ck\rangle$. A natural basis for $\Qsym_k$ is $\{M_\alpha\}_{\alpha\in\Ck}$.

\begin{Theorem} 
The Hopf algebras $\Qsymk$ and $\Nsymk$  are dual. Under this duality, the basis $\{M_\alpha\}_{\alpha \in \Ck}$ is dual to $\{H_\alpha\}_{\alpha \in \Ck}$.
\end{Theorem}

\begin{proof}
Follows in the same way that $\symkk$ and $\symk$ are realized as dual Hopf algebras.
\end{proof}

We let $\langle \cdot, \cdot \rangle_{\Ck}: \Qsymk \times \Nsymk \rightarrow \mathbb{Q}$ denote the pairing inherited from the pairing $\langle \cdot, \cdot \rangle_{C} : \Qsym \times \Nsym$ .  In particular, $\langle M_\alpha, H_\beta \rangle_{\Ck} = \delta_{\alpha, \beta}$ for $\alpha, \beta \in \Ck$.

\subsection{Horizontal $k$-composition strips and quasi-symmetric affine Schur functions}

For two compositions $\beta \lessdot_{\C} \alpha$, we say that $\alpha // \beta$ is a horizontal $k$-composition strip if $\alpha// \beta$ is a composition strip and $\lambda(\alpha)/\lambda(\beta)$ is a horizontal $k$-strip.

A semistandard $k$-composition tableaux of shape $\alpha$ and content $\beta$ is a sequence of $k$-bounded compositions $\emptyset = \gamma^0 \lessdot_{\C} \gamma^1 \lessdot_{\C} \gamma^2 \lessdot_{\C} \cdots \lessdot_{\C} \gamma^m = \alpha$ such that the skew diagrams $\gamma^i // \gamma^{i-1}$ are horizontal $k$-composition strips of size $\beta_i$. The number of semistandard $k$-composition tableaux of shape $\alpha$ and content $\beta$ is denoted $\widetilde{K}^{(k)}_{\alpha, \beta}$. 

\begin{Example} Let $k = 3$, $\alpha = [1,3,1,1]$ and $\beta = [1,1,2,1,1]$. Then there are two semistandard $k$-composition tableaux of shape $\alpha$ and content $\beta$. They are \[\emptyset \lessdot_{\C} [1] \lessdot_{\C} [1,1] \lessdot_{\C} [2,1,1] \lessdot_{\C} [3,1,1] \lessdot_{\C} [1,3,1,1] \textrm{ and} \] \[\emptyset \lessdot_{\C} [1] \lessdot_{\C} [1,1] \lessdot_{\C} [2,1,1] \lessdot_{\C} [1,2,1,1] \lessdot_{\C} [1,3,1,1].\]

\[\tikztableau{{$5$},{$3$, $3$,$4$},{$2$}, {$1$}} \hspace{.7in} \tikztableau{{$4$},{$3$, $3$,$5$},{$2$}, {$1$}}\]

\end{Example}

\begin{Theorem}\label{def:nsk}
There exists a basis, which we will denote $\{\NSk\}_{\alpha \in \Ck}$, of $\Nsymk$ uniquely defined by the non-commutative affine Pieri rule: 
\[ H_i \NSk  = \sum_\beta \NSk[\beta],\] where the sum is over all compositions $\beta$ for which $\beta // \alpha$ is a horizontal $k$-composition strip of size $i$.
\end{Theorem}

\begin{proof}
It is sufficient to show that the non-commutative affine Pieri rule is invertible. We place an ordering on compositions of $n$ as defined in \cite{HLMvW}. The partial order $\alpha < \beta$ when $\lambda(\alpha) < \lambda(\beta)$ in dominance order is extended to a linear order. Then one can see that $H_i \NSk = \NSk[{[i] \cup \alpha}] + \sum_{\beta < [i] \cup \alpha} \NSk[\beta]$. Therefore the transition matrix between $\{H_i \NSk \}_{i \in \{1, \dots k\}, \alpha \vDash n-i}$ and $\{ \NSk[\beta]\}_{\beta \vDash n }$ is unitriangular and hence $\NSk$ is well defined and forms a basis of $\Nsymk$.
\end{proof}

\begin{Corollary}\label{rem:htonckschur}
As a consequence of the non-commutative affine Pieri rule, \[H_\beta = \sum_{\alpha} \widetilde{K}^{(k)}_{\alpha, \beta} \NSk .\]
\end{Corollary}

\begin{Theorem}\label{thm:ktoinf}
As $k$ approaches infinity, \[ \NSk \to \NS.\] 
\end{Theorem}

\begin{proof}
For $k$ large enough, the condition of $\lambda(\alpha)/ \lambda(\beta)$ being a horizontal $k$-strip in Theorem \ref{def:nsk} becomes vacuous, so the definition matches Definition \ref{def:ncschur}.
\end{proof}

\begin{Definition}
The quasi-symmetric affine Schur functions $\{\QSk\}_{\alpha \in \Ck}$ are the basis of $\Qsymk$ which are dual to the basis of non-commutative affine Schur functions.
\end{Definition}

\begin{Theorem}
The quasi-symmetric affine Schur functions expand in the quasi-symmetric monomial basis as: \[\QSk = \sum_{\beta} \widetilde{K}^{(k)}_{\alpha, \beta} M_\beta.\]
\end{Theorem}

\begin{proof}

Using  Corollary \ref{rem:htonckschur}, \[ \langle \QSk, H_\beta \rangle_{\Ck} = \langle \QSk, \sum_{\gamma} \widetilde{K}^{(k)}_{\gamma, \beta} \NSk[\gamma] \rangle_{\Ck} = \widetilde{K}^{(k)}_{\alpha, \beta},\] so the coefficient of $M_\beta$ in $\QSk[\alpha]$ is $\widetilde{K}^{(k)}_{\alpha, \beta}$.
\end{proof}

\begin{Theorem}
As $k$ approaches infinity, \[ \QSk \to \QS.\] 
\end{Theorem}

\begin{proof}
Follows from Theorem \ref{thm:ktoinf} and duality.
\end{proof}

\begin{Remark}
The map $\chi: \Nsym \to \sym$ is well defined on the sub algebra $\Nsymk$, and maps onto $\symk$.
\end{Remark}

Contrasting with Theorem \ref{projection}, we see that non-commutative affine Schur functions are pre images of $k$-Schur functions under $\chi$.

\begin{Theorem} For a $k$-bounded composition $\alpha$, 
\[ \chi(\NSk) = s^{(k)}_{\lambda(\alpha)}.\]
\end{Theorem}

\begin{proof}
This follows from induction, Theorem \ref{def:nsk} and the fact that $\alpha // \beta$ being a horizontal composition strip implies that $\lambda(\alpha)/\lambda(\beta)$ is a horizontal strip, a condition already imposed by the fact that $\lambda(\alpha)/\lambda(\beta)$ is a horizontal $k$-strip.
\end{proof}

Contrasting with Theorem \ref{decompose}, we see that the quasi-symmetric affine functions decompose the 
dual $k$-Schur functions into quasi-symmetric components.

\begin{Theorem}
For a $k$-bounded partition $\lambda$, \[ \dualkschur = \sum_{\substack{\alpha \\ \lambda(\alpha) = \lambda}} \QSk.\]
\end{Theorem}

\begin{proof}
We note that if $f \in \Qsymk$ is actually in $\symkk$ and $g \in \Nsymk$ then $\langle f, g \rangle_{\Ck} = \langle f, \chi(g) \rangle_{\Pk}$.
We can compute $\langle \dualkschur, \NSk \rangle_{\Ck} = \langle \dualkschur, s^{(k)}_{\lambda(\alpha)} \rangle_{\Pk} = \delta_{\lambda, \lambda(\alpha)}$, so the quasi-symmetric $\QSk$ function appears in $\dualkschur$ with coefficient $\delta_{\lambda, \lambda(\alpha)}$.
\end{proof}

\begin{Remark}
The elements $\NSk$ are the unique elements in $\Nsymk$ for which the Pieri rule is determined by the Pieri rule for $\NS$ and the $k$-Pieri rule for $\kschur_\lambda$. The curious reader will be compelled to ask if their structure coefficients are positive and if $\NSk$ expands positively in the $\NS$ basis. Computer exploration has determined neither holds.
\end{Remark}

\section*{Appendix}

We give explicit expansions of $\NSk$ in the non-commutative complete homogenous basis and $\QSk$ in the monomial quasi-symmetric basis for $k = 2, 3$ and $k$-bounded compositions $\alpha \vDash n$ for $2 \leq n \leq 4$. We hope to make our code available in Sage-Combinat \cite{sage-combinat} soon.

For $k = 2$ and $n=2$, the change of basis matrix from $\NSk$ to $H_\beta$ is:

\[
\bordermatrix{
~ & [2] & [1, 1] \cr
[2] & 1 & 0 \cr
[1, 1] & -1 & 1 \cr
}
\]
For $k = 2$ and $n=3$, the change of basis matrix from $\NSk$ to $H_\beta$ is:
\[
\bordermatrix{
~ & [2, 1] & [1, 2] & [1, 1, 1] \cr
[2, 1] & 1 & 0 & 0 \cr
[1, 2] & 0 & 1 & 0 \cr
[1, 1, 1] & 0 & -1 & 1 \cr
}
\]
For $k = 2$ and $n=4$, the change of basis matrix from $\NSk$ to $H_\beta$ is:
\[
\bordermatrix{
~ & [2, 2] & [2, 1, 1] & [1, 2, 1] & [1, 1, 2] & [1,
1, 1, 1] \cr
[2, 2] & 1 & 0 & 0 & 0 & 0 \cr
[2, 1, 1] & -1 & 1 & 0 & 0 & 0 \cr
[1, 2, 1] & -1 & 0 & 1 & 0 & 0 \cr
[1, 1, 2] & -1 & 0 & 0 & 1 & 0 \cr
[1, 1, 1, 1] & 1 & -1 & 0 & -1 & 1 \cr
}
\]
For $k = 3$ and $n=2$, the change of basis matrix from $\NSk$ to $H_\beta$ is:
\[
\bordermatrix{
~ & [2] & [1, 1] \cr
[2] & 1 & 0 \cr
[1, 1] & -1 & 1 \cr
}
\]
For $k = 3$ and $n=3$, the change of basis matrix from $\NSk$ to $H_\beta$ is:
\[
\bordermatrix{
~ & [3] & [2, 1] & [1, 2] & [1, 1, 1] \cr
[3] & 1 & 0 & 0 & 0 \cr
[2, 1] & -1 & 1 & 0 & 0 \cr
[1, 2] & -1 & 0 & 1 & 0 \cr
[1, 1, 1] & 1 & -1 & -1 & 1 \cr
}
\]
For $k = 3$ and $n=4$, the change of basis matrix from $\NSk$ to $H_\beta$ is:
\[
\bordermatrix{
~ & [3, 1] & [1, 3] & [2, 2] & [2, 1, 1] & [1, 2, 1]
& [1, 1, 2] & [1, 1, 1, 1] \cr
[3, 1] & 1 & 0 & 0 & 0 & 0 & 0 & 0 \cr
[1, 3] & 0 & 1 & 0 & 0 & 0 & 0 & 0 \cr
[2, 2] & 0 & -1 & 1 & 0 & 0 & 0 & 0 \cr
[2, 1, 1] & 0 & 0 & -1 & 1 & 0 & 0 & 0 \cr
[1, 2, 1] & 0 & 0 & -1 & 0 & 1 & 0 & 0 \cr
[1, 1, 2] & 0 & 0 & -1 & 0 & 0 & 1 & 0 \cr
[1, 1, 1, 1] & 0 & 1 & 0 & 0 & -1 & -1 & 1
\cr
}
\]
For $k = 2$ and $n=2$, the change of basis matrix from $\QSk$ to $M_\beta$ is:
\[
\bordermatrix{
~ & [2] & [1, 1] \cr
[2] & 1 & 1 \cr
[1, 1] & 0 & 1 \cr
}
\]
For $k = 2$ and $n=3$, the change of basis matrix from $\QSk$ to $M_\beta$ is:
\[
\bordermatrix{
~ & [2, 1] & [1, 2] & [1, 1, 1] \cr
[2, 1] & 1 & 0 & 0 \cr
[1, 2] & 0 & 1 & 1 \cr
[1, 1, 1] & 0 & 0 & 1 \cr
}
\]
For $k = 2$ and $n=4$, the change of basis matrix from $\QSk$ to $M_\beta$ is:
\[
\bordermatrix{
~ & [2, 2] & [2, 1, 1] & [1, 2, 1] & [1, 1, 2] & [1,
1, 1, 1] \cr
[2, 2] & 1 & 1 & 1 & 1 & 1 \cr
[2, 1, 1] & 0 & 1 & 0 & 0 & 1 \cr
[1, 2, 1] & 0 & 0 & 1 & 0 & 0 \cr
[1, 1, 2] & 0 & 0 & 0 & 1 & 1 \cr
[1, 1, 1, 1] & 0 & 0 & 0 & 0 & 1 \cr
}
\]
For $k = 3$ and $n=2$, the change of basis matrix from $\QSk$ to $M_\beta$ is:
\[
\bordermatrix{
~ & [2] & [1, 1] \cr
[2] & 1 & 1 \cr
[1, 1] & 0 & 1 \cr
}
\]
For $k = 3$ and $n=3$, the change of basis matrix from $\QSk$ to $M_\beta$ is:
\[
\bordermatrix{
~ & [3] & [2, 1] & [1, 2] & [1, 1, 1] \cr
[3] & 1 & 1 & 1 & 1 \cr
[2, 1] & 0 & 1 & 0 & 1 \cr
[1, 2] & 0 & 0 & 1 & 1 \cr
[1, 1, 1] & 0 & 0 & 0 & 1 \cr
}
\]
For $k = 3$ and $n=4$, the change of basis matrix from $\QSk$ to $M_\beta$ is:
\[
\bordermatrix{
~ & [3, 1] & [1, 3] & [2, 2] & [2, 1, 1] & [1, 2, 1]
& [1, 1, 2] & [1, 1, 1, 1] \cr
[3, 1] & 1 & 0 & 0 & 0 & 0 & 0 & 0 \cr
[1, 3] & 0 & 1 & 1 & 1 & 1 & 1 & 1 \cr
[2, 2] & 0 & 0 & 1 & 1 & 1 & 1 & 2 \cr
[2, 1, 1] & 0 & 0 & 0 & 1 & 0 & 0 & 0 \cr
[1, 2, 1] & 0 & 0 & 0 & 0 & 1 & 0 & 1 \cr
[1, 1, 2] & 0 & 0 & 0 & 0 & 0 & 1 & 1 \cr
[1, 1, 1, 1] & 0 & 0 & 0 & 0 & 0 & 0 & 1 \cr
}
\]

\section*{Acknowlegements}

The authors have benefited from conversations and computations of Tom Denton and Franco Saliola.

This research was facilitated by computer exploration using the open-source
mathematical software system \texttt{Sage}~\cite{sage} and its algebraic
combinatorics features developed by the \texttt{Sage-Combinat} community
\cite{sage-combinat}.

\bibliographystyle{halpha}
\bibliography{references} 

\end{document}